\documentclass[12pt]{amsart}
\pdfoutput=1
\usepackage{amssymb}
\usepackage{amsmath}
\usepackage{amsfonts}
\usepackage[usenames]{color}
\usepackage{graphicx}
\usepackage{array}
\usepackage{psfrag}
\usepackage{color}
\usepackage{ulem}

\makeatletter
 
 \@addtoreset{equation}{section}
\makeatother

\newtheorem{theorem}{Theorem}
\newtheorem{lemma}[theorem]{Lemma}

\theoremstyle{definition}

\newtheorem{observation}[theorem]{Observation}

\newtheorem{corollary}[theorem]{Corollary}

\theoremstyle{remark}
\newtheorem{remark}[theorem]{Remark}

\numberwithin{equation}{section}

\newcommand{\bibtitle}[1]{\textit{#1}}

\begin{document}

\title{ Complete minors of self-complementary graphs}

\date{\today}
\author{Andrei Pavelescu}
\address{University of South Alabama, Mobile, AL 36688}
\email{andreipavelescu@southalabama.edu}

\author{Elena Pavelescu}
\address{University of South Alabama, Mobile, AL 36688}
\email{elenapavelescu@southalabama.edu}

\maketitle

\begin{abstract}
We show that any self-complementary graph  with $n$ vertices contains a $K_{\lfloor \frac{n+1}{2}\rfloor}$ minor.
We derive topological properties of self-complementary graphs.
\end{abstract}
\vspace{0.1in}

\section{Introduction}

There are interesting connections between the topological properties of a graph and those of its complement. 
For example, in \cite{BHK} it is shown that the complement of a planar graph with nine vertices is non planar.
In particular, any self-complementary graph on nine vertices is non planar.
A self-complementary graph is a graph which is isomorphic to its complement. 
Results in \cite{KLV} and \cite{HLS} on the $\mu-$invariant introduced by C. de Verdi\`{e}re \cite{dV} imply that  the complement of a planar graph on ten vertices is intrinsically linked. 
This means that every embedding of such complement in $\mathbb{R}^3$ contains two disjoint linked cycles.

These topological properties of graphs are connected to the existence of complete minors. 
For a graph $G$, \textit{a minor of G} is any graph that can be obtained from $G$ by a sequence of edge deletions and contractions.
An edge contraction means identifying its endpoints and deleting any loops and double edges thus created.
Complete  minors of graphs have been studied extensively.
The order of the largest complete minor of a graph $G$ is called the Hadwiger number of $G$, $h(G)$.
The Hadwiger conjecture, one of the famous problems in graph theory, states that for any graph $G$, $\chi(G)\le h(G)$, where $\chi(G)$ is the chromatic number of $G$.
In this paper, we search for complete minors of self-complementary graphs and we prove the  following
 
 \begin{theorem}
 \label{sc} Any self-complementary graph $G$ on $n \ge 1$ vertices contains a $K_{\lfloor \frac{n+1}{2}\rfloor}$ minor.
 \end{theorem}
 
 We give examples proving that this bound is the best possible.
 These results add to the existing body of work on self-complementary graphs started with seminal papers by H. Sachs \cite{Sa} and G. Ringel \cite{R}. 
 We note that the ${\lfloor \frac{n+1}{2}\rfloor}$ lower bound was independently found in \cite{RS}.
The last section presents connections between Hadwiger numbers and some topological properties of self-complementary graphs: outerplanarity, planarity, intrinsic linkness, intrinsic knottedness.
  


\section{Notation and Background on Self-Complementary Graphs}

In this article, all graphs are non-oriented, without loops  and without multiple edges. 
For a graph $G$ with $n$ vertices, $V(G)=\{v_1, v_2, \ldots , v_n\}$ denotes the set of vertices of $G$, and $E(G)=\{v_iv_j | \ v_i  \text{ is connected to } v_j\}$ denotes the set of edges  of $G$.
The complete graph with $n$ vertices is denoted by $K_n$.
For $G$ a graph with $n$ vertices,  $cG$ denotes the \textit{complement of $G$} in $K_n$. 
The graph $cG$ has the same set of vertices as $G$ and $E(cG) = \{ v_iv_j | v_iv_j\notin E(G)\}$.
A graph $G$ is called \textit{self--complementary} if there exists a graph isomorphism $\rho:G\rightarrow cG$.
The isomorphism $\rho$ can be viewed as a permutation on the set of vertices of $G$.
For a subset of vertices $\{v_{i_1}, v_{i_2}, \ldots , v_{i_k}\} $, we let  $\big< v_{i_1}, v_{i_2}, \ldots , v_{i_k}\big>_G$ denote the subgraph of $G$ induced by this set.
For two graphs $G_1$ and $G_2$, $G_1\ast G_2$ represents the graph defined by $V(G_1\ast G_2)=V(G_1)\sqcup V(G_2)$ and $E(G_1\ast G_2)=E(G_1)\sqcup E(G_2)\sqcup E$, with $E=\{ab| a\in V(G_1), \, b\in V(G_2)\}$. For example, $K_6=K_4*K_2$.\\

The complete graph on $n$ vertices has $n(n-1)/2$ edges.
Since a self-complementary graph with $n$ vertices has $n(n-1)/4$ edges, 
such a graph has either $n=4k$ or $n=4k+1$ vertices.
In \cite{Sa2}, Sachs describes the cycle structure of $\rho$ for a general self-complementary graph.

\begin{theorem}[H. Sachs]
\label{sachs} Let $G$ be a self-complementary graph on $n$ vertices and $\rho : G \rightarrow cG$ be a fixed isomorphism. Then 
\begin{itemize}
\item[(a)] if $n=4k$, then $\rho$ has no fixed points and all the cycles of $\rho$ have lengths divisible by 4;
\item[(b)] if $n=4k+1$, then $\rho$ has exactly one fixed point and all the nontrivial cycles of $\rho$ have lengths divisible by 4.
\end{itemize}
\end{theorem}
Vertices $a,b \in V(G)$ are said to be \textit{similar} if $b=\rho^{2i}(a)$ for some $i$. 
Since the restriction of a graph isomorphism is an isomorphism onto its image, for $S \subset V(G)$ such that $\rho(S)=S$, $\big< S\big>_G $ is itself self-complementary.

\section{Construction of Complete Minors}

In this section, we prove the existence of complete minors of self-complementary graphs. We begin with self-complementary  graphs on $4n$ vertices and extend the result to self-complementary  graphs on $4n+1$ vertices.

Let $G$ be a self-complementary  graph with $4n$ vertices with a fixed isomorphism $\rho : G \rightarrow cG$. 
In particular, via a labelling of the vertices, $\rho$ is an element of $S_{4n}$, the group of permutations on $4n$ characters. 
Since, $deg_G(\rho(v))=4n-1-deg_{cG}(\rho(v))=4n-1-deg_G(v)$, the permutation $\rho$ takes vertices of $G$ of degree larger than the average $2n-\frac{1}{2}$  to vertices of $G$ of degree smaller than the average, and vice versa. 
This means that exactly half of the vertices of $G$ have degree at least $2n$. 

Let $N=\big<v_1,v_2,\ldots, v_{2n}\big>_G$ denote the subgraph of $G$ generated by vertices with degree at least $2n$ and let $H=\big<v_{2n+1}, v_{2n+2},\ldots,v_{4n}\big>_G$ denote the subgraph of $G$ generated by vertices with degree at most $2n-1$.
The isomorphism $\rho$ takes the vertices of $N$ to the vertices of $H$ and the vertices of $H$ to the vertices of  $N$.
Moreover, if $v_iv_j$ is an edge of $N$, then $\rho(v_i)\rho(v_j)$ is an edge of $cG$ and thus is not an edge of $H$.
It follows that $N$ and $H$ are isomorphic to graphs which are complements of each other in $K_{2n}$.

Let $L$ denote the subgraph of $G$ whose edge set is the union of all edges with one endpoint in $N$ and the other in $H$.
Since $\rho$ interchanges vertices of $N$ and $H$, for an edge $e$ in $L$,  $\rho(e)$ is an edge in $cG$ which has one endpoint in $H$ and the other in $N$.
This means $L$ is isomorphic to $\rho(L)$ and their union is $K_{2n,2n}$. 
We say $L$ is \textit{bipartite self-complementary in $K_{2n,2n}$ (BSC)}.
It follows that  $|E(L)|=2n^2$. 

\begin{observation}
\label{cyclehalf}
Since any cycle of $\rho$ is set-invariant,  the graph induced by the vertices of this cycle is self--complementary. 
If $\rho$ is a cycle of length $4k$, then $2k$ of its vertices belong to $N$ and the other $2k$ belong to $H$.
Thus the automorphism $\rho$ induces a bipartite self-complementary graph in $K_{2k,2k}$,
 which has $2k^2$ edges. By similarity of vertices within $N$ and $H$ respectively, each vertex connects to $k$ neighbors. 
 \end{observation}
 
 With the notation of this section, we have:

\begin{lemma}
\label{4ncycle} Let $G$ be a self-complementary graph on $4n$ vertices. 
Assume $\rho : G \rightarrow cG$ is a $4n-$cycle. 
Then $G$ contains a $K_{2n}$ minor obtained by contracting $2n$ pairwise nonadjacent edges of $L$.
\end{lemma}

\begin{proof}
Let $\rho=(a, \rho(a)...\rho^{4n-1}(a))$,  $V(N)=\{a, \rho^2(a),...,\rho^{4n-2}(a)\}$ and $V(H)=\{\rho(a), \rho^3(a),...,\rho^{4n-1}(a)\}$. 
Let $1 \le t \le 4n-1$ odd such that $a\rho^t(a) \in E(L)$. 
By Observation \ref{cyclehalf}, we have $n$ such choices for $t$.
For $0\le i \le 2n-1$, since $\rho^{2i}$ is a graph isomorphism, $\rho^{2i}(a)\rho^{2i+t}(a) \in E(L)$.


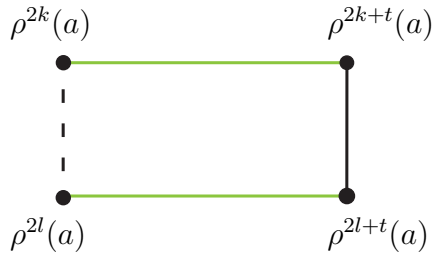
\begin{figure}[htpb!]
\begin{center}
\begin{picture}(170, 85)
\put(0,0){\includegraphics[width=2in]{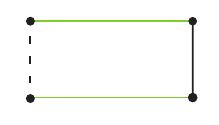}}
\put(0,0){$\rho^{2l}(a)$}
\put(120,0){$\rho^{2l+t}(a)$}
\put(0,80){$\rho^{2k}(a)$}
\put(120,80){$\rho^{2k+t}(a)$}
\end{picture}
\caption{Green edges are edges of $L$ which are contracted to obtain the $K_{2n}$ minor.}
\label{fig-diagram4n}
\end{center}
\end{figure}

For $0 \le k < l \le 2n-1$ such that $\rho^{2k}(a)\rho^{2l}(a) \notin E(N)$, by applying $\rho^t$, it follows that $\rho^{2k+t}(a)\rho^{2l+t}(a) \in E(H)$. See Figure \ref{fig-diagram4n}.
This shows that the contraction of the $2n$ edges $\rho^{2i}(a)\rho^{2i+t}(a),\, 0\le i \le 2n-1$ produces a $K_{2n}$ minor of $G$. \end{proof}

\begin{observation}
\label{consecutive} If $a$ is a vertex of a self-complementary graph $G$ with $4n$ vertices, since $\rho$ reverses incidence relations, $a$ neighbors exactly one of $\{\rho(a), \rho^{-1}(a)\}$. In particular, in any cycle of $\rho$ one may choose a generator $a$ such that $a\rho(a)\in E(G)$.
\end{observation}

\begin{lemma}
\label{2cycles} 
Let $G$ be a self-complementary graph on $4n$ vertices. 
Assume $\rho : G \rightarrow cG$ factors as the product of two cycles of lengths $4k$ and $4l=4(n-k)$, respectively. 
Then $G$ contains a $K_{2n}$ minor obtained by contracting $2n$ pairwise nonadjacent edges of $L$.
\end{lemma}
\begin{proof}

Let $\rho=\tau\sigma=(a\rho(a)...\rho^{4k-1}(a))(b\rho(b)...\rho^{4l-1}(b))$. By Observation \ref{consecutive}, we may assume $a\rho(a), b\rho(b) \in E(L)$.

 
\begin{figure}[htpb!]
\begin{center}
\begin{picture}(170, 85)
\put(0,0){\includegraphics[width=2in]{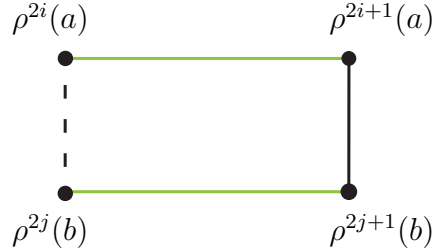}}
\put(0,0){$\rho^{2j}(b)$}
\put(120,0){$\rho^{2j+1}(b)$}
\put(0,80){$\rho^{2i}(a)$}
\put(120,80){$\rho^{2i+1}(a)$}
\end{picture}
\caption{ There is always and edge between the pair of vertices $\{\rho^{2i}(a), \rho^{2i+1}(a)\}$ and the pair of vertices $\{\rho^{2j}(b), \rho^{2j+1}(b)\}$}
\label{fig-diagram-mix}
\end{center}
\end{figure}

Let $0\le i\le 2k-1$ and $0 \le j \le 2l-1$. 
If $\rho^{2i}(a)\rho^{2j}(b) \notin E(G)$, then, by applying $\rho$, we obtain $\rho^{2i+1}(a)\rho^{2j+1}(b)\in E(G)$. 
See Figure \ref{fig-diagram-mix}.
This proves that there is at least one edge between $\{\rho^{2i}(a),\rho^{2i+1}(a)\}$ and $\{\rho^{2j}(b),\rho^{2j+1}(b)\}$, for any $0\le i\le 2k-1$ and $0 \le j \le 2l-1$. 
This and Lemma \ref{4ncycle}  imply that the minor obtained by contracting the edges $\rho^{2i}(a)\rho^{2i+1}(a)$ for $0 \le i \le 2k-1$ and the edges $\rho^{2j}(b)\rho^{2j+1}(b)$ for $0 \le j \le 2l-1$ is isomorphic to $K_{2k}\ast K_{2l}=K_{2n}$.
\end{proof}

\begin{lemma} Let $G$ be a self-complementary graph on $4n$ vertices. Then $G$ contains a $K_{2n}$ minor obtained by contracting $2n$ pairwise nonadjacent edges of $L$.

\label{sc4n}
\end{lemma}
 
 \begin{proof}
 Let $\rho : G \rightarrow cG$ be a fixed isomorphism. 
 Let $\rho= \tau_1\tau_2...\tau_s$ be the cycle decomposition, where the length of each $\tau_i$ is a multiple of 4. We shall prove the conclusion by induction on $s$, the number of disjoint cycles of $\rho$.
 The base case is done by Lemma \ref{4ncycle}.
 For $s>1$, by Observation \ref{consecutive}, for each $1 \le i \le s$ there is a vertex $a_i$ such that $\tau_i=(a_i\rho(a_i)...)$ and $a_i\rho(a_i) \in E(L)$. 
For any cycle $\tau_1 \ne \tau_k$ of $\rho$, the arguments of Lemma \ref{2cycles} show that there is at least one edge between  $\{\rho^{2i}(a_1),\rho^{2i+1}(a_1)\}$ and $\{\rho^{2j}(a_k),\rho^{2j+1}(a_k)\}$, for any $0\le i, j$. By the induction hypothesis, we obtain a $K_{2n-ord(\tau_1)/2}$ minor of the subgraph of $G$ induced by the vertices not in $\tau_1$.  Each of the vertices of this minor connects to at least one vertex of each pair $\{\rho^{2i}(a_1),\rho^{2i+1}(a_1)\}$, $0\le i< ord(\tau_1)$. Contracting these $\frac{1}{2}ord(\tau_1)$ edges, we obtain a $K_{2n}=K_{2n-ord(\tau_1)/2}\ast K_{ord(\tau_1)/2}$ minor of $G$.
 \end{proof}


\begin{lemma} If $G$ is a self-complementary graph with $4n+1$ vertices, then $G$ contains a $K_{2n+1}$ minor.
\label{sc4n+1}
\end{lemma}
\begin{proof}
Let $G$ be a self-complementary graph whose set of vertices is $V(G)=\{v_1,...,v_{4n+1}\}$. 
Let $\rho :G \rightarrow cG$ be a fixed isomorphism. 
By Theorem \ref{sachs}, $\rho$ has exactly one fixed point. Assume $\rho(v_{4n+1})=v_{4n+1}$.
The graph $\overline{G}=\big <v_1, ...,v_{4n} \big >_G$ is a self-complementary graph with $4n$ vertices. 
Let $N=\big <v_1, ...,v_{2n} \big >_{\overline{G}}$, $H=\big <v_{2n+1}, ...,v_{4n} \big >_{\overline{G}}$. We assumed $deg_{\overline{G}}(v_i)\ge 2n$ for $1\le i \le 2n$ and $deg_{\overline{G}}(v_i)\le 2n-1$ for $2n+1 \le i \le 4n$. 

For any $a\in V(\overline{G})$, as $v_{4n+1}$ is a fixed point for $\rho$ and $\rho$ reverses incidence relations, $v_{4n+1}$ connects to exactly one of $a$ or $\rho(a)$. This means that for any nontrivial cycle of $\rho$, $v_{4n+1}$ neighbors exactly half of the vertices of this cycle. Moreover, these neighbors are either all vertices  of $N$, or all vertices of $H$. By Lemma \ref{sc4n}, $\overline{G}$ contains a $K_{2n}$ minor which can be obtained by contracting $2n$ pairwise nonadjacent edges of $L$.
 Since  $v_{4n+1}$ connects to exactly one endpoint of each of these edges, the contraction of the $2n$ edges creates a $ K_{2n+1}=K_{2n}\ast K_1$ minor of $G$. 
 \end{proof}
 We summarize the results in lemmas \ref{sc4n} and \ref{sc4n+1} in the following: 
 
 \begin{theorem}
 \label{sc} Any self-complementary graph $G$ on $n \ge 1$ vertices contains a $K_{\lfloor \frac{n+1}{2}\rfloor}$ minor.\\
 \end{theorem}

The bound $\lfloor \frac{n+1}{2}\rfloor$ is the best that can be guaranteed as the next theorem proves.

\begin{observation}
\label{sharp} 
\begin{enumerate}
\item[(a)]  For $n \ge 1$, there exist self-complementary graphs with $4n$ vertices which do not contain a $K_{2n+1}$ minor.
\item[(b)] For $n \ge 0$, there exist self-complementary graphs with $4n+1$ vertices which do not contain a $K_{2n+2}$ minor.
\end{enumerate}
\end{observation}
\begin{proof}
\begin{enumerate}
\item[(a)] Let $N\simeq K_{2n}$ such that $V(N)=\{v_1,...,v_{2n}\}$ and $H$ such that $V(H)=\{w_1,...,w_{2n}\}$ and $E(H)=\emptyset$. Let $L=L_1 \sqcup L_2$, with $L_1$ the complete bipartite graph on $\{v_1,...,v_n\}$ and $\{w_1,...,w_n\}$, and $L_2$  the complete bipartite graph on $\{v_{n+1},...,v_{2n}\}$ and $ \{w_{n+1},...,w_{2n}\}$. The graph $G=N \cup L \cup H$ is a self-complementary graph on $4n$ vertices. Assume that $G$ contains a $K_{2n+1}$ minor. Since $|V(N)|=2n$ and $H$ is a discrete set of points with no edges between its vertices, there exists $w_i\in V(H)$ such that none of the edges adjacent to $w_i$ were contracted in creating this minor. The degree of $w_i$ in the created minor is at most $deg_G(w_i)=n$, which is a contradiction as the degree of any vertex in $K_{2n+1}$ is $2n$.

\item[(b)] Let $G'$ be the self-complementary graph on $4n+1$ vertices obtained by adding one vertex $a$  and  all edges $av_i$  to the graph $G$ described in part (a). 
The same argument as in part $(a)$  shows that $G'$ cannot contain a $K_{2n+2}$ minor.
\end{enumerate}
\end{proof}

\begin{remark}The presence of complete minors in self-complementary graphs enables us to revisit a conjecture about the $\mu-$invariant introduced by C. de Verdi\`{e}re in \cite{dV}. For any simple graph $G$, $\mu(G)$ is the largest corank of a class of matrices defined based on incidence relations in $G$. In  \cite{dV},  de Verdi\`{e}re proved that $\mu$ is monotone under taking minors. It was conjectured in \cite{KLV} that for any graph $G$ on $n$ vertices, $\mu(G)+\mu(cG)\ge n-2$. One can check, using only the definition of $\mu$, that $\mu(K_n)=n-1$, for all $n\ge 2$. For a self-complementary graph $G$ on $4n$ vertices, Theorem \ref{sc} shows that $\mu(G)+\mu(cG)\ge 2n-1+2n-1= 4n-2$. For a graph $G$ on $4n+1$ vertices, Theorem \ref{sc} shows that $\mu(G)+\mu(cG)\ge 2n+2n= 4n > (4n+1)-2$. This validates the conjecture for self-complementary graphs and shows that equality does not always hold.
\end{remark}


\section{Topological Properties of Self-Complementary Graphs}
In this section we discuss topological properties of self-complementary graphs which  derive from Theorem \ref{sc}: outerplanarity,
planarity, intrinsic linkness and intrinsic knottedness. 
A graph is called \textit{outerplanar} if it can be embedded in the plane such that all its vertices sit on the outer face of the embedding.
It is known that a graph is outerplanar if and only if it doesn't have  $K_4$ or $K_{2,3}$ as its minors. 
A graph is called \textit{planar} if it can be embedded in the plane.
The two excluded minors for planarity are $K_5$ and $K_{3,3}$ \cite{Ku}.

A graph $G$ is said to be \textit{intrinsically linked }(IL) if every embedding of $G$ in $\mathbb{R}^3$ contains two disjoint linked cycles.
  If not IL, then $G$ is \textit{linklessly embeddable} (nIL).
 By work of Conway and Gordon \cite{CG}  and Robertson, Seymour, and Thomas \cite{Sa}, a graph is intrinsically linked if and only if it contains one of the seven graphs in the Petersen family as a minor.
 One of these graphs is $K_6$. 
 A graph $G$ is said to be \textit{intrinsically knotted }(IK) if every embedding of $G$ in $\mathbb{R}^3$ contains a nontrivial knot.
 By work of Conway and Gordon \cite{CG}, we know that $K_7$ is IK. 
A graph is said to be $n-$\textit{apex} if deleting $n$ or fewer vertices produces a planar (sub)graph.
By \cite{Sa}, a $1-$apex graph is nIL.
By \cite{BBFFHL}, a $2-$apex graph is not IK.

\begin{corollary}
\label{outerplanar}
A self-complementary graph on $n\ge 8$ vertices in not outerplanar.  
\end{corollary}
\begin{proof}
By Theorem \ref{sc}, any such graph $G$ contains a $K_4$ minor, thus $G$ is not outerplanar.
\end{proof}
 We note that all  self-complementary  graphs on less than 8 vertices are outerplanar. See Figure \ref{scsmall}(a).
 While this is a corollary to our theorem, it is not a new result.
A  self-complementary graph on $n\ge 8$ vertices has at least $2n-2$ edges, and it is known that such graphs are not outerplanar. 


\begin{figure}[htpb!]
\begin{center}
\begin{picture}(400, 70)
\put(0,0){\includegraphics[width=5in]{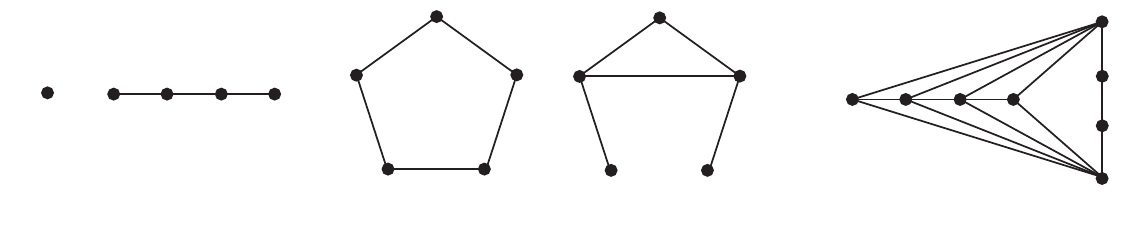}}
\put(130,0){(a)}
\put(305,0){(b)}
\end{picture}
\caption{(a) Self-complementary graphs on one, four and five vertices, (b) Planar self-complementary graph on 8 vertices }
\label{scsmall}
\end{center}
\end{figure}

\begin{corollary}
\label{planar}
A self-complementary graph on $n\ge 9$ vertices in not planar.
\end{corollary}
\begin{proof}
By Theorem \ref{sc}, any such graph $G$ contains a $K_5$ minor and thus it is not planar.
\end{proof}
While non-planarity of self-complementary graphs with $n\ge9$ vertices is also a consequence of \cite{BHK}, we are establishing the existence of a $K_5$ minor. 
Figure \ref{scsmall}(b).

\begin{corollary}
A self-complementary graph on $n\ge 12$ vertices is intrinsically linked.
\end{corollary}
\begin{proof}
By Theorem \ref{sc}, any such graph $G$ contains a $K_6$ minor. 
Since $K_6$ is an excluded minor for linkless embbedability, it follows that $G$ is intrinsically linked.
\end{proof}

\begin{corollary}
A self-complementary graph on $n\ge 13$ vertices in intrinsically knotted.
\end{corollary}
\begin{proof}
By Theorem \ref{sc}, any such graph $G$ contains a $K_7$ minor. 
Since $K_7$ is an excluded minor for knotless embbedability, it follows that $G$ is intrinsically knotted.
\end{proof}

On the other hand, there are self-complementary graphs on 12 vertices which are not intrisically knotted.
We describe the construction of such a graph.
Figure \ref{notIK}(a) represents a bipartite self-complementary graph in $K_{6,6}$.
Construct a self-complementary graph on 12 vertices $G$ by adding all edges between filled vertices. 
Deleting vertices $a$ and $b$ of $G$ creates the planar graph in Figure \ref{notIK}(b).
This means $G$ is $2-$apex and thus it is not intrinsicallly knotted.


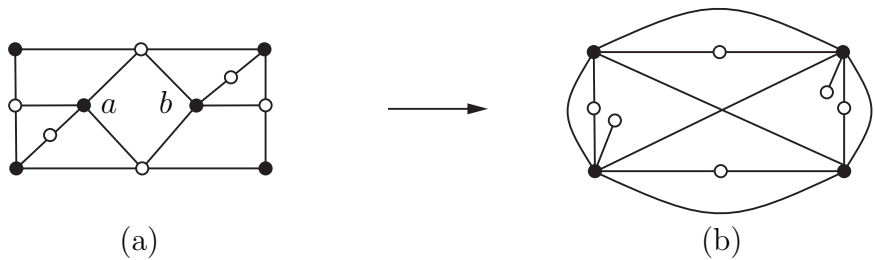
\begin{figure}[htpb!]
\begin{center}
\begin{picture}(370, 90)
\put(0,0){\includegraphics[width=5in]{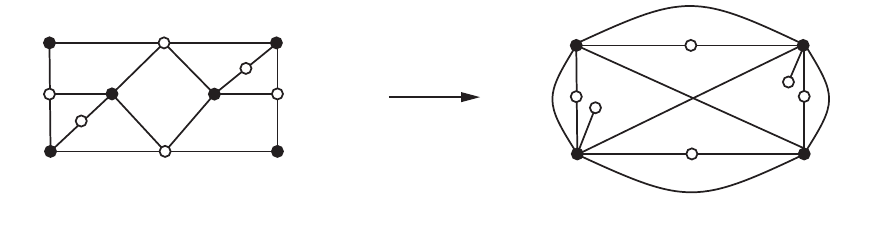}}
\end{picture}
\put(-310,5){(a)}
\put(-90,5){(b)}
\put( -317,57){$a$}
\put(-295,57){$b$}
\caption{(a) A self-complementary graph on 12 vertices is obtained by pairwise connecting all filled vertices, (b) Removing vertices $a$ and $b$ gives a planar subgraph. }
\label{notIK}
\end{center}
\end{figure}

\bibliographystyle{amsplain}

\end{document}